\documentclass[12pt,reqno]{amsart}

\numberwithin{equation}{section} 

\usepackage{enumerate}
\usepackage{amssymb}
\usepackage{amsmath}
\usepackage{amscd}
\usepackage{amsthm}
\usepackage{amsfonts}
\usepackage{graphicx}
\usepackage[all]{xy}
\usepackage{verbatim}
\usepackage{hyperref}

\newtheorem{theorem}{Theorem}[section]

\newtheorem{lemma}[theorem]{Lemma}
\newtheorem{corollary}[theorem]{Corollary}
\newtheorem{definition}[theorem]{Definition}

\newtheorem{conjecture}[theorem]{Conjecture}

\newtheorem{problem}[theorem]{Problem}

\newtheorem{question}[theorem]{Question}

\newcommand{\al}{\alpha}
\newcommand{\be}{\beta}
\newcommand{\ga}{\gamma}
\newcommand{\Ga}{\Gamma}
\newcommand{\del}{\delta}

\newcommand{\Lam}{\Lambda}
\newcommand{\eps}{\epsilon}

\newcommand{\sig}{\sigma}

\newcommand{\om}{\omega}
\newcommand{\Om}{\Omega}
\newcommand{\vphi}{\varphi}

\newcommand{\cO}{\mathcal{O}}

\newcommand{\cS}{\mathcal{S}}

\newcommand{\bC}{\mathbb{C}}
\newcommand{\bR}{\mathbb{R}}
\newcommand{\bZ}{\mathbb{Z}}
\newcommand{\bQ}{\mathbb{Q}}

\newcommand{\goa}{\mathfrak{a}}

\newcommand\set[1]{\left\{#1\right\}}
\newcommand\pa[1]{\left(#1\right)}

\newcommand\idist[1]{\langle#1\rangle}
\newcommand\av[1]{\left|#1\right|}

\newcommand{\onto}{\xymatrix{\ar@{>>}[r]&}}
\newcommand{\da}[4]{\xymatrix{#1 \ar@<.5ex>[r]^{#2} \ar@<-.5ex>[r]_{#3} & #4}}

\begin{document}
\title{A solution to a problem of Cassels and Diophantine properties of cubic numbers}
\author{Uri Shapira}
\thanks{* Part of the author's Ph.D thesis at the Hebrew University of Jerusalem.\\ Email: ushapira@gmail.com}
\begin{abstract}
We prove that almost any pair of real numbers $\al,\be$, satisfies the following inhomogeneous uniform version of Littlewood's conjecture:  
\begin{equation}\label{main equation}
\forall \ga,\del\in\bR,\quad \liminf_{|n|\to\infty} \av{n}\idist{n\al-\ga}\idist{n\be-\del}=0,
\end{equation}
where $\idist{\cdot}$ denotes the distance from the nearest integer. The existence of even a single pair that satisfies \eqref{main equation}, solves a problem of Cassels ~\cite{Ca} from the 50's. We then prove that if  $1,\al,\be$ span a totally real cubic number field, then $\al,\be$, satisfy \eqref{main equation}. This generalizes a result of Cassels and Swinnerton-Dyer, which says that such pairs satisfies Littlewood's conjecture. It is further shown that if $\al,\be$ are any two real numbers, such that $1,\al,\be$, are linearly dependent over $\bQ$, they cannot satisfy ~\eqref{main equation}. The results are then applied to give examples of irregular orbit closures of the diagonal groups of a new type. The results are derived from rigidity results concerning hyperbolic  actions of higher rank commutative groups on homogeneous spaces. 
\end{abstract}
\maketitle
\section{Introduction}\label{introduction}
\subsection{Notation} We first fix our notation and define the basic objects to be discussed in this paper.
Let $X_d$ denote the space of $d$-dimensional unimodular lattices in $\bR^d$ and let $Y_d$ denote the space of translates of such lattices. Points of $Y_d$ will be referred to as \textit{grids}, hence for $x\in X_d$ and $v\in\bR^d$, $y=x+v\in Y_d$ is the grid obtained by translating the lattice $x$ by the vector $v$. We denote by $\pi$  the natural projection 
\begin{equation}\label{projections}
 Y_d \stackrel{\pi}{\longrightarrow} X_d,\quad x+v  \mapsto x.
 \end{equation} 
For each $x\in X_d$, we identify the fiber $\pi^{-1}(x)$ in $Y_d$ with the torus $\bR_d/x$. 
 Let $N:\bR^d\to \bR$ denote the function $N(w)=\prod_1^dw_i$.
For a grid $y\in Y_d$, we define the \textit{product set} of $y$ to be
\begin{equation}\label{prod set}
P(y)=\set{N(w) : w\in y}
\end{equation}
In this paper we shall study properties of the product set. We will mainly be interested in density properties and the values near zero. We denote
\begin{equation}\label{mu}
N(y)= \inf\set{\av{N(w)} : w\in y }.
\end{equation}
The ambiguous use of the symbol $N$ both for a function on $\bR^d$ and for a function on $Y_d$ and the lack of appearance of the dimension $d$ in the notation should not cause any confusion. 
The \textit{inhomogeneous minimum} of a lattice $x\in X_d$ is defined by
\begin{equation}\label{inhomogeneous minimum}
\mu(x)=\sup\set{N(y) : y\in \pi^{-1}(x)}.
\end{equation}  
The \textit{inhomogeneous Markov spectrum} (or just the spectrum) is defined by
\begin{equation}\label{spectrum}
\cS_d=\set{\mu(x) : x\in X_d}.
\end{equation}   
A more geometric way to visualize the above notions is the following: The \textit{star body} of radius $\eps>0$ is the set $S_\eps=\set{w\in\bR^d : \av{N(w)}<\eps}$. In terms of star bodies, for a grid $y\in Y_d$, $N(y)=\inf\set{\eps : S_\eps\cap y\ne\emptyset}$ and for a lattice $x\in X_d$, $\mu(x)$ is the least number such that for any $\eps>\mu(x)$, the star body $S_\eps$ intersects all the grids of $x$ or equivalently, $S_\eps$ projects onto the torus $\pi^{-1}(x)$ under the natural projection.
\subsection{Dimension 2}\label{dim2} In ~\cite{D} Davenport showed (generalizing a result of Khintchine) that for any $x\in X_2$ one has $\mu(x)>\frac{1}{128}$, hence the spectrum $\cS_2$ is bounded away from zero. The constant $\frac{1}{128}$ is not optimal and much work has been done to improve Davenport's lower bound of the spectrum (see \cite{Ca},\cite{Ca2},\cite{D} and the references therein). The set $\set{y\in\pi^{-1}(x): \mu(y)>0}$ (where $x\in X_2$ is arbitrary) has been investigated too and in a recent work ~\cite{Ts} it was shown that it has full Hausdorff dimension and in fact that it is a winning set for Schmidt's game. 
\subsection{Cassels problem} In his book \cite{Ca} (p. 307), Cassels raised the following natural question:
\begin{problem}[Cassels]\label{Cassels problem}
In dimension $d\ge 3$, is the infemum of the spectrum $\cS_d$ equal to zero.
\end{problem}
We answer Cassels' problem affirmatively. In fact we show that the infemum is attained and give explicit constructions of lattices attaining the minimum. The following theorem is a consequence of corollary ~\ref{almost any GDP}(1):
\begin{theorem}\label{theorem 101}
For $d\ge 3$, almost any lattice $x\in X_d$ (with respect to the natural probability measure) satisfies $\mu(x)=0$.
\end{theorem}
\subsection{Diophantine approximations}\label{DA}
Of particular interest to Diophantine approximations, are lattices of the following forms: Let $v\in\bR^{d-1}$ be a column vector. Denote 
\begin{equation}\label{hor sgrps}
h_v=\pa{
\begin{array}{ll}
I_{d-1}&v\\
0 & 1
\end{array}
},\quad 
g_v=\pa{
\begin{array}{ll}
1&v^t\\
0 & I_{d-1}
\end{array}
}
\end{equation}
where $I_{d-1}$ denotes the identity matrix of dimension $d-1$ and the $0$'s denote the corresponding trivial vectors.
Let $x_{v},z_v\in X_d$, denote the lattices spanned by the columns of $h_v$ and $g_v$ respectively. For $\ga\in\bR$, denoting  by $\idist{\ga}$, the distance from $\ga$ to the nearest integer, an easy calculation shows that the statements
\begin{equation}\label{mu(x)1}
\forall \vec{\ga}\in\bR^{d-1}\quad \liminf_{|n|\to\infty} \left|n\right|\prod_1^{d-1}\idist{nv_i-\ga_i}=0,
\end{equation}
\begin{equation}\label{mu(z)2}
\forall\ga\in \bR\quad \liminf_{\prod \av{n_i}\to\infty}\prod_1^{d-1}\av{n_i}\idist{\sum_1^{d-1}n_iv_i-\ga}=0,	
\end{equation}
imply that $\mu(x_{v})=0$ and $\mu(z_{v})=0$ respectively.
\begin{theorem}\label{corollary 102} Let $d\ge 3$
\begin{enumerate}
\item For almost any $v\in\bR^{d-1}$ (with respect to Lebesgue measure) ~\eqref{mu(x)1} and ~\eqref{mu(z)2} are satisfied and in particular $\mu(x_{v})=\mu(z_{v})=0$.
\item Nonetheless, if $\dim span_\bQ\set{1,v_1,\dots,v_{d-1}}\le 2$ then $\mu(x_{v}),\mu(z_{v})$ are positive.
\end{enumerate}
\end{theorem}
Part (1) of the above theorem is a consequence of corollary ~\ref{almost any GDP}. Part (2) follows from known results in dimension 2 and will be proved in \S\ref{irregular orbits}.\\
Perhaps the most interesting amongst the results in this paper is the following theorem which shows that certain pairs of algebraic numbers are generic. The proof follows from corollary ~\ref{cubics are good} and corollary ~\ref{almost any GDP}(3).
\begin{theorem}\label{my result}
 If $1,\al,\be$ form a basis for a totally real cubic number field, then  
 \begin{equation}\label{my result 1}
\forall \ga,\del\in\bR\quad \liminf_{|n|\to\infty} |n|\idist{n\al-\ga}\idist{n\be-\del}=0,
\end{equation} 
 \begin{equation}\label{my result 1'}
\forall \ga\in\bR\quad \liminf_{|nm|\to\infty} |nm|\idist{n\al+m\be-\ga}=0.
\end{equation}
\end{theorem}
\subsection{Remarks}
\begin{enumerate}
\item Cassels and Swinnerton-Dyer have shown \cite{CaSD} that any real pair $\al,\be$, belonging to the same cubic totally real field, satisfies Littlewood's conjecture, i.e. satisfies ~\eqref{my result 1} with $\ga=\del=0.$ Thus theorem ~\ref{my result} together with theorem ~\ref{corollary 102}(2), can be viewed as a strengthening of their result.  
\item As Cassels points out in his book \cite{Ca}, problem ~\ref{Cassels problem} belongs to a family of problems for various forms (other than $N$). Barnes \cite{Ba} solved an analogous problem with $N$ replaced by an indefinite quadratic form in $d\ge 3$ variables. Our method, when adapted appropriately,  seem to give a different proof of Barnes' result.
\item In a recent paper \cite{Bu}, Y. Bugeaud raised (independently of Cassels) the question of existence of pairs $\al,\be\in\bR$ which satisfy ~\eqref{main equation}.
\item Our methods are dynamical and rely on rigidity results such as Ratner's theorem ~\cite{R}, the results and techniques appearing in ~\cite{LW} and the extension of Furstenberg's times 2 times 3 theorem ~\cite{F} due to Berend ~\cite{B}. But, although the usual ergodic theoretic arguments provide existence only, our results provide us with concrete examples of numbers and lattices with nontrivial dynamical and Diophantine properties.   
\end{enumerate}
\subsection{Acknowledgments}: I would like to express my gratitude to Elon Lindenstrauss for numerous useful conversations, especially for the reference to ~\cite{Ca} and for discussing theorem ~\ref{density}. I thank Hillel Furstenberg and Yitzhak Katznelson for stimulating conversations concerning problem ~\ref{Cassels problem}. Finally, I would like to express my deepest gratitude and appreciation to my adviser and teacher, Barak Weiss, for his constant help, encouragement and belief.  
\section{Basic notions, groups and homogeneous spaces}\label{Basic notions}
When $d\ge 2$ is fixed we denote $G=SL_d(\bR)\ltimes\bR^d$, $G_0=SL_d(\bR)$ and $V=\bR^d$. We shall identify $G_0,V$ with the corresponding subgroups of $G$. Denote by $A<G_0$ the subgroup of diagonal matrices with positive diagonal entries. The Lie algebra of $A$ is identified with the Euclidean $d-1$ dimensional space
\begin{equation}\label{euclidean space}
\goa=\set{\textbf{t}=(t_1,\dots,t_d)\in\bR^d : \sum_1^dt_i=0}.
\end{equation}
$A$ is isomorphic to the additive group $\goa$, via the exponent map $\exp:\goa\to A$ given by $\exp(\textbf{t})=diag(e^{t_1},\dots,e^{t_d}).$ We denote the inverse of $\exp$ by $\log$. The \textit{roots} of $A$ are the linear functionals on $\goa$ of the following forms
\begin{equation}\label{roots of A}
\forall 1\le i\ne j\le d, \quad \textbf{t}\mapsto t_i-t_j ; \quad\forall 1\le k\le d,\quad \textbf{t}\mapsto t_k.
\end{equation}
The set of roots will be denoted by $\Phi$. As suggested in ~\eqref{roots of A}, we say that a root $\al\in\Phi$ corresponds to a pair $1\le i\ne j\le d$ or to an index $1\le k\le d$. To each root $\al\in\Phi$, there corresponds a one parameter unipotent subgroup $\set{u_\al(t)}_{t\in\bR}<G$, called the \textit{root group}, for which the following equation is satisfied
\begin{equation}\label{alpha(a)}
au_\al(t)a^{-1}=u_\al(e^{\al(\log(a))}t).
\end{equation} 
When the root $\al$ corresponds to a pair $i\ne j$, we sometime denote $u_\al(t)=u_{ij}(t)$. In this case $u_{ij}(t)\in G_0$ is the matrix all of whose entries are zero, except for the $ij$'th which is equal to $t$ and the diagonal entries which are equal to $1$. When $\al$ corresponds to $1\le k\le d$, we sometime denote $u_\al(t)=u_k(t)$. In this case, $u_k(t)\in V$ is the vector, $te_k$, where $e_k$ is the $k$'th standard vector. We sometime abuse notation and write, for a root $\al\in\Phi$ and $a\in A$, $\al(a)$ instead of $\al(\log(a))$.

For an element $a\in A$ we define the \textit{stable horospherical subgroup} of $G$ corresponding to it to be $U^-(a)=\left\{(g,v)\in G : a^n(g,v)a^{-n}\to_{n\to\infty} e\right\}$, and the \textit{unstable horospherical subgroup} to be $U^+(a)=U^-(a^{-1})$. An element $b\in A$ is called \textit{regular} if for any root $\al\in\Phi$, $\al(b)\ne 0$. For $b\in A$, any element $g\in G$ which is close enough to $e$, has a unique decomposition $g=cu^+u^-$, where $c$ centralizes $b$, $u^+\in U^+(b),u^-\in U^-(b)$ and $c,u^+,u^-$ lie in corresponding neighborhoods of $e$. If $b$ is regular then the centralizer of $b$ is $A$.\\
The linear action of $G_0$ on $\bR^d$ induces a transitive action of $G_0$ on $X_d$. The stabilizer of the lattice $\bZ^d\in X_d$ is $\Ga_0=SL_d(\bZ)$. This enables us to identify $X_d$ with the homogeneous space $G_0/\Ga_0$. For $g\in G_0$, we denote $\bar{g}=g\Ga_0$. $\bar{g}\in X_d$ represents the lattice spanned by the columns of the matrix $g$. In a similar manner we identify $Y_d$ with $G/\Ga$, where $\Ga=SL_d(\bZ)\ltimes \bZ^d$. For $(g,v)\in G$,  $(g,v)\Ga$ represents the grid $\bar{g}+v.$ $\Ga$ (resp $\Ga_0$) is a lattice in $G$ (resp $G_0$). The $G$ (resp $G_0$) invariant probability measure on $Y_d$ (resp $X_d$) will be referred to as the \textit{Haar measure}. $G_0$ and its subgroups act on $X_d,Y_d$ and the action commutes with the projection $\pi:Y_d\to X_d$. Finally, we say that a grid $y=x+v$ is \textit{rational}, if $v$ belongs to the $\bQ$-span of the lattice $x$. This is equivalent to saying that $y\in\pi^{-1}(x)$ is a torsion element.
\section{Compact $A$ orbits}\label{compact orbits} The following classification theorem, essentially goes back to ~\cite{Bac}. A modern proof can be found in ~\cite{LW} or ~\cite{Mc}. Before stating it, let us recall some notions from number theory. A \textit{totally real number field} is a finite extension of $\bQ$, all of whose embeddings into $\bC$ are real. A \textit{lattice} in a number field, is the $\bZ$-span of a basis of the field over $\bQ$. Let $K$ be a totally real number field of degree $d$ and let $\sig_i, i=1\dots d$, be the different embeddings of $K$ into the reals. The map $\vphi=(\sig_1,\dots,\sig_d)^t:K\to\bR^d$ is called a \textit{geometric embedding}. It is well known that if $\Lam$ is a lattice in $K$, then $\vphi(\Lam)$ is a lattice in $\bR^d.$ The ring of integers in $K$ is denoted by $\cO_K$ and the group of units of this ring is denoted by $\cO_K^*$. The logarithmic embedding of $\cO_K^*$ in $\goa$ (see ~\eqref{euclidean space}) is given by $\om\mapsto(\log\av{\sig_1(\om)},\dots,\log\av{\sig_d(\om)}).$ We shall denote the image of $\cO_K^*$ by $\Om_K$. Dirichlet's unit theorem implies that $\Om_K$ is a lattice in $\goa$. 
\begin{theorem}\label{compact lattice orbits}
Let $x_0\in X_d$. $Ax_0$ is compact if and only if there exists $a\in A$ such that $ax_0$ is (up to multiplication by a normalizing scalar) the geometric embedding of a lattice in a totally real number field, $K$, of degree $d$. Moreover there exists some finite index subgroup $\Om<\Om_K$ such that $\Om=\log\pa{Stab_A(x_0)}$.
\end{theorem} 
As a corollary we get a classification of the compact $A$-orbits in $Y_d$. The proof is left to the reader.
\begin{corollary}\label{compact grid orbits}
A grid $y\in Y_d$ has a compact $A$-orbit, if and only if it is rational and $\pi(y)\in X_d$ has a compact orbit. In this case, $Stab_A(y)$ is of finite index in $Stab_A(\pi(y))$.
\end{corollary}
The following corollary of theorem \ref{compact lattice orbits} is one of the places in which higher rank is reflected. 
\begin{corollary}\label{independence}
Let $d\ge 3$ and let $y\in Y_d$ be a grid with a compact $A$-orbit. Denote $A_0=Stab_A(y)$. Then, for any root $\al\in\Phi$, the set $\set{\al(a) : a\in A_0}$ is dense in the reals.
\end{corollary}
\begin{proof} Let $K$ be the totally real number field of degree $d$ arising from theorem ~\ref{compact lattice orbits} and corollary ~\ref{compact grid orbits} and let $\al\in\Phi$ be a root. By corollary ~\ref{compact grid orbits}, $\log(A_0)$ is of finite index in $\Om_K$. It follows that it is enough to justify why $\al(\Om_K)$ is dense in the reals. As $\Om_K$ is a lattice in $\goa$, this is equivalent to $\Om_K\cap ker(\al)$ not being a lattice in $ker(\al)$. If $\al$ corresponds to a pair $i\ne j$ (see ~\eqref{roots of A}), then if $\Om_K\cap ker(\al)<ker(\al)$ is a lattice, then there is a subfield of $K$ (the field $\set{\theta\in K : \sig_i(\theta)=\sig_j(\theta)}$) with a group of units containing a copy of $\bZ^{d-2}$. The degree of this subfield is at most $d/2$ and so by Dirichlet's unit theorem the degree of the group of units in this subfield is at most $d/2-1$. This means that $d/2-1\ge d-2$ which is equivalent to $d\le 2$ a contradiction. If $\al$ corresponds to $k$, then the situation is even simpler as $\Om_K\cap ker(\al)=\set{0}$.  
\end{proof}
\section{Dynamics and \textit{GDP} lattices}\label{existence}
\subsection{Inheritance} The reason that the action of $A$ on $X_d,Y_d$ is of importance to us is the invariance of the product set, namely  $\forall a\in A, y\in Y_d$, $P(y)=P(ay)$.
\begin{definition}\label{GDP definition}\quad
\begin{enumerate}
	\item A grid $y\in Y_d$ is called $DP$ (dense products) if $\overline{P(y)}=\bR$.
	\item A lattice $x\in X_d$ is called $GDP$ if any grid $y\in\pi^{-1}(x)$ is $DP$.
\end{enumerate}
\end{definition} 
The proofs of the next useful lemma and its corollary are left to the reader.
\begin{lemma}[Inheritance]\label{inheritance lemma}
If $y,y_0\in Y_d$ are such that $y_0\in\overline{Ay}$, then $\overline{P(y_0)}\subset \overline{P(y)}$.
\end{lemma}
\begin{corollary}\label{inheritance 2}\quad
\begin{enumerate}
	\item If $y,y_0\in Y_d$ are such that $y_0\in\overline{Ay}$ and $y_0$ is $DP$, then $y$ is $DP$ too.
	\item If $x,x_0\in X_d$ are such that $x_0\in\overline{Ax}$ and $x_0$ is $GDP$, then $x$ is $GDP$ too.
\end{enumerate}
\end{corollary}
\begin{lemma}\label{two ways to prove DP}\quad
\begin{enumerate}
	\item If $y\in Y_d$ is such that $\overline{Ay}\supset\pi^{-1}(x_0)$ for some $x_0\in X_d$ then $y$ is $DP$.
	\item If $y\in Y_d$ is such that there exists $y_0\in Y_d$ and a root group $\set{u_{ij}(t)}_{t\in\bR}<G_0$ such that  $\overline{Ay}\supset \set{u_{ij}(t)y_0 : t\in I}$, where $I\subset \bR$ is a ray, then $y$ is $DP$. 
\end{enumerate}
\end{lemma}
\begin{proof}
To see (1) note that from the inheritance lemma it follows that $\forall v\in\bR^d$ $P(x_0+v)\subset \overline{P(y)}$. Clearly $\cup_{v\in\bR^d}P(x_0+v)=\bR$. To see (2) note that it follows from ~\cite{R} theorem B, that $$\set{u_{ij}(t)y_0:t\in\bR}\subset \overline{\set{u_{ij}(t)y_0 : t\in I}}.$$
Let $w\in y_0$ be a vector all of whose coordinates are nonzero. By the inheritance lemma  $$\overline{P(y)}\supset\set{N\pa{u_{ij}(t)w} : t\in \bR}=\set{N(w)\pa{\frac{w_j}{w_i}t+1}:t\in\bR}=\bR.$$
\end{proof}
\subsection{Existence of $GDP$ lattices for $d\ge 3$} The proof of the following theorem is based on the ideas presented in ~\cite{LW}.
\begin{theorem}\label{existence of GDP}
If $x,x_0\in X_d$ ($d\ge 3$), $Ax_0$ is compact and $x_0\in \overline{Ax}\setminus Ax$, then $x$ is $GDP$.
\end{theorem}
\begin{proof}
Let $y\in \pi^{-1}(x)$. Consider $F=\overline{Ay}$ and $F_0=F\cap\pi^{-1}(x_0)$. Note that from the compactness of the fibers of $\pi$ and the assumptions of the theorem it follows that $F_0\ne\emptyset$. In ~\cite{Sh} (lemma 4.8) it is shown that any irrational grid $y_0\in \pi^{-1}(x_0)$, satisfies $\overline{Ay_0}\supset \pi^{-1}(x_0)$. Hence by lemma ~\ref{two ways to prove DP} (1), $y_0$ is $DP$. Hence, if $F_0$ contains an irrational grid then $y$ is $DP$ by corollary ~\ref{inheritance 2} (1).\\
Assume then that $F_0$ contains only rational grids and let $y_0\in F_0$ (this could happen for example if $y$ is a rational grid).
By corollary ~\ref{compact grid orbits}, $Ay_0$ is compact. Denote $A_0=Stab_A\pa{y_0}.$ 
Choose a regular element $b\in A_0$. Let $U^-=U^-(b),U^+=U^+(b),$ be the corresponding stable and unstable horospherical subgroups of $G$. Any point which is close enough to $y_0$ in $Y_d$, has a unique representation of the form $au^+u^-y_0$, where $a\in A,u^+\in U^+$ and $u^-\in U^-$ are in corresponding neighborhoods of the identity. Choose a sequence $y_n\to y_0$ from the orbit $Ay$. We may assume that 
\begin{equation}\label{transverse convergence}   
 y_n=a_nu_n^+u_n^-y_0 \in F
\end{equation}
where $a_n,u_n^+,u_n^-\to e$. We may further assume that $a_n=e$ for all $n$, for if not, replace $y_n$ by $a_n^{-1}y_n$. The fact that $y_0$ is not in $Ay$ implies that the pairs $(u_n^+,u_n^-)$ are nontrivial for any $n$. Our first goal is to show:\\
\textbf{Claim} 1: \textit{There exist a point in $F$ of the form $uy_0$, where $u\ne e$ is in $U^+$ or $U^-$.}\\
If there exists an $n$ with one of $u_n^+$ or $u_n^-$, being trivial, then the claim follows. If not, we denote for any $n$ by $k_n$, the least integer such that the maximum of the absolute values of the entries of $b^{k_n}u_n^+b^{-k_n}$ is greater than $1$. It then follows that this absolute value lies in some interval of the form $[1,M]$ (where $M$ only depends on the choice of $b$). Since $u_n^+\to e$ we must have $k_n\to\infty$. It is easy to see that the convergence $b^{k_n}u^-b^{-k_n}\to e$, for $u^-\in U^-$, is uniform on compact subsets of $U^-$. Hence in particular, $b^{k_n}u_n^-b^{-k_n}\to e$. Thus after going to a subsequence and abusing notation, we may assume that $b^{k_n}u_n^+b^{-k_n}\to u$, where $e\ne u\in U^+.$ Hence
$$\lim b^{k_n}y_n=\lim b^{k_n}u_n^+b^{-k_n}b^{k_n}u_n^-b^{-k_n}y_0=uy_0\in F$$
and claim 1 follows.\\
\textbf{Claim} 2: \textit{There exist a root $\al\in \Phi$ and $t_0\ne 0$ such that $u_\al(t_0)y_0\in F$.}\\
Let $u$ be as in claim 1. We denote for $g\in G$, $$\Phi_g=\set{\al\in\Phi :\textrm{ the entry corresponding to $\al$ in $g$ is nonzero}}.$$ If $\Phi_u$ contains only one root, claim 2 follows. If not, there exists a one parameter semigroup $\{a_t\}_{t\ge 0}<A$ such that $\Phi_u$ is the union of two non empty disjoint sets, $\Phi_u^-,\Phi_u^0$, such that for $\al\in\Phi_u^-$, $\al(a_1)<0$, while for $\al\in\Phi_u^0$, $\al(a_1)=0$ (see ~\cite{LW} step 4.5 for details). It follows that for any sequence $t_n\to\infty$, $a_{t_n}ua_{t_n}^{-1}\to u'$, where $\Phi_{u'}=\Phi_u^0$, which is strictly smaller then $\Phi_u$. Since $Ay_0\simeq A/A_0$ is a $(d-1)$-torus, we can always find a sequence $t_n\to\infty$, such that $a_{t_n}y_0\to y_0$. Thus $$\lim a_{t_n}uy_0=\lim a_{t_n}ua_{t_n}^{-1}a_{t_n}y_0=u'y_0\in F.$$
Repeating this process a finite number of times, we end up with a root $\al$ and some nonzero real number $t_0$, such that $u_\al(t_0)y_0\in F$ and claim 2 follows.\\
\textbf{Claim} 3: \textit{ There exists a ray $I\subset \bR$ such that $\set{u_\al(t)y_0 : t\in I}\subset F$}.\\
By corollary ~\ref{independence}, we have that $\set{\al(a) : a\in A_0}$ is dense in $\bR$. It follows that $I=\overline{\set{e^{\al(a)}t_0 : a\in A_0}}$ is a ray. We have 
$$\set{au_\al(t_0)y_0 : a\in A_0} = \set{au_\al(t_0)a^{-1}y_0 : a\in A_0} = \set{u_\al(e^{\al(a)}t_0)y_0 : a\in A_0}\subset F.$$
Claim 3 now follows from the fact that $F$ is closed.\\
Note that from our assumption that $F_0$ contain only rational grids, it follow that the root group $u_\al(t)$ is contained in $G_0$. It now follows from lemma ~\ref{two ways to prove DP} (2) that $y$ is $DP$ and the theorem follows.
\end{proof}
\begin{corollary}\label{dense imply GDP}
For $d\ge 3$, any lattice with a dense $A$ orbit is $GDP$.
\end{corollary}
\begin{proof}
This is a consequence of theorem ~\ref{existence of GDP}, and corollary ~\ref{inheritance 2}(2).
\end{proof}
The following lemma is well known. We give the outline of a proof.
\begin{lemma}\label{expanding leaf generic}
For any $d\ge 2$ and almost any $v\in \bR^{d-1}$ (with respect to Lebesgue measure) $\overline{Ax_{v}}=\overline{Az_{v}}=X_d$.
\end{lemma}
\begin{proof}
Let us consider lattices of the form $x_v$ for example. Denote $$a_t=diag(e^t,\dots,e^t,e^{(1-d)t}).$$ Note that for any positive $t$ the unstable horospherical subgroup of $a_t$ is (recall the notation of $\S\S$~\ref{DA}) $U^+(a_t)=\set{h_v:v\in\bR^{d-1}}$.  For any point $x\in X_d$ there exists neighborhoods $W^+_x,W^-_x,W^0_x$ of the identity elements in the groups $U^+(a_t),U^-(a_t)$ and the centralizer of $a_t$, such that the map $W^0_x\times W^-_x\times W^+_x\to X_d$ given by $(c,g,h_v)\mapsto cgh_vx$ is a diffeomorphism with a neighborhood, $W_x$, of $x$ in $X_d$. Note that if $x_i=c_ig_ih_vx, i=1,2$ are two points in $W_x$ having the same $U^+$ coordinate, then the trajectory $\set{a_tx_1}_{t\ge 0}$ is dense in $X_d$ if and only if $\set{a_tx_2}_{t\ge 0}$ is dense too.
As the action of $a_t$ on $X_d$ is ergodic we know that for almost any $x'\in W_x$, $\set{a_tx'}_{t\ge 0}$ is dense in $X_d$ and from analyzing the structure of the Haar measure on $X_d$ restricted to $W_x$ we conclude that for almost any $v$ in the neighborhood of zero in $\bR^{d-1}$ corresponding to $W^+_x$, $\set{a_th_vx}_{t\ge 0}$ is dense in $X_d$. We abuse notation and think of $W_x^+$ as contained in $\bR^{d-1}$.\\
To finish the argument we find a countable collection $v_i\in\bR^{d-1}$ such that the neighborhoods $W_{x_{v_i}}$ satisfy $\bR^{d-1}=\cup_i\pa{v_i+W_{x_{v_i}}}$ and note that $W_{x_{v_i}}x_{v_i}=\set{x_w:w\in v_i+W_{x_{v_i}}}.$
\end{proof}
\begin{corollary}\label{almost any GDP}
Let $d\ge 3$
\begin{enumerate}
	\item  Almost any lattice $x\in X_d$ (with respect to Haar measure) is $GDP$.
	\item For almost any $v\in\bR^{d-1}$ (with respect to Lebesgue measure), both $x_{v},z_{v}\in X_d$ are $GDP$.
	\item If $v\in\bR^{d-1}$ is such that $x_v$ (resp $z_v$) is $GDP$, then ~\eqref{mu(x)1} (resp ~\eqref{mu(z)2}) is satisfied.
\end{enumerate}
\end{corollary}
\begin{proof}
(1) follows from the ergodicity of the $A$ action on $X_d$ which in particular means that almost any point has a dense orbit and corollary ~\ref{dense imply GDP}. (2) follows from lemma ~\ref{expanding leaf generic} and corollary ~\ref{dense imply GDP}. (3) is left to be verified by the reader.
\end{proof}
\section{A density result}\label{density result}
Let $x_0\in X_3$ be a point with a compact $A$-orbit. We shall use the following facts: It follows from Lemma 4.1 of ~\cite{LW}, that the orbit of $x_0$ under any root group $u_{ij}(t)$, is dense in $X_3$, moreover Theorem B of ~\cite{R}, implies that in fact $\set{u_{ij}(t)x_0}_{t\in I}$ is dense in $X_3$, for any ray $I\subset \bR$. It follows from corollary 1.4 in ~\cite{LW}, that if $b\in G_0$ is lower or upper triangular but not diagonal, then $\overline{Abx_0}= X_3$. A more careful look yields the following theorem. The author is indebted to Elon Lindenstrauss, for valuable ideas appearing in the proof.
\begin{theorem}\label{density}
Let $x_0\in X_3$ be a lattice with a compact $A$-orbit. 
If 
$p= \pa{
\begin{array}{lll}
\al&0&0\\
\be&\ga&\del\\
\eta&\tau&\mu
\end{array}
}\in G_0$ is such that both $\tau,\mu\ne 0$, then $\overline{Apx_0}= X_3$.  
\end{theorem}  
\begin{proof}
A straightforward computation shows
\begin{equation}\label{shrinking translate}
p=\pa{
\begin{array}{lll}
1&0&0\\
0&1&\frac{\del}{\mu}\\
0&0&1
\end{array}
}
\pa{
\begin{array}{lll}
\al&0&0\\
\be-\frac{\del\eta}{\mu} & \ga-\frac{\del\tau}{\mu}&0\\
\eta&0&\mu
\end{array}
}
\pa{
\begin{array}{lll}
1&0&0\\
0& 1&0\\
0&\frac{\tau}{\mu}&1
\end{array}
}=u_{23}(t_0)b_1b_2,
\end{equation}
where we denoted $t_0=\frac{\del}{\mu}$ and the matrices appearing in the middle of ~\eqref{shrinking translate} by $u_{23}\pa{t_0},b_1$ and $b_2$ according to appearance. Note that the matrix $b=b_1b_2$ is nondiagonal as $\tau\ne 0$, hence by the preceding discussion, if we denote $x_1=bx_0$, then $x_1$ has a dense $A$-orbit. Hence, it is enough to show that $x_1$ belongs to the orbit closure of $px_0=u_{23}\pa{t_0}x_1$. This will follow from the existence of a recurrence sequence $a_n\in A$ for $x_1$ (i.e. a sequence such that $a_nx_1\to x_1$) which in addition satisfies $a_nu_{23}\pa{t_0}a_n^{-1}\to e,$ for then 
\begin{equation}\label{no name}
\lim a_n px_0=\lim a_n u_{23}(t_0)a_n^{-1}a_nx_1=x_1.
\end{equation}
A sequence $a_n$ satisfies $a_nu_{23}(t_0)a_n^{-1}\to e$, if and only if $t_2^{(n)}-t_3^{(n)}\to-\infty$, where $\textbf{t}^{(n)}=\log(a_n)$. Thus it is enough to show that for any $m>0$, there exists a recurrence sequence for $x_1$  in $A_m=\exp\pa{R_m}$, where $R_m=\set{\textbf{t}\in \goa : t_2-t_3\le-m}$ is a half plane. Choose $m>0$. We shall show that in fact $A_mx_1$ is dense in $X_3$. Denote 
\begin{equation}\label{a-t}
a_t=diag\pa{e^{2t},e^{-t},e^{-t}}\textrm{ and }a'=diag\pa{e^{-m},1,e^m}.
\end{equation}
The line $\set{a'a_t}_{t\in\bR}$ lies on the boundary of $A_m$. As $b=b_1b_2$, we write (emphasizing the desired partition into products)
\begin{equation}\label{c.1}
A_mx_1\supset a'a_tx_1=a'a_tbx_0=\pa{a'a_tb_1\pa{a'a_t}^{-1}}\cdot\pa{\pa{a'a_t}b_2\pa{a'a_t}^{-1}}\cdot\pa{\pa{a'a_t}x_0}.
\end{equation}
We observe that for any sequence $t_n\to\infty$, $a'a_{t_n}b_1\pa{a'a_{t_n}}^{-1}$ converges to the diagonal matrix $a''=diag\pa{\al,\ga-\frac{\del\tau}{\mu},\mu}$, while at the same time $\pa{a'a_{t_n}}b_2\pa{a'a_{t_n}}^{-1}$ converges to $u_{32}\pa{s_0},$  where $s_0=e^m\frac{\tau}{\mu}\ne 0$. Furthermore, because $Ax_0\simeq A/Stab_A(x_0)$, and because the line $a_t$ is irrational with respect to the lattice $Stab_A(x_0)$ (by theorem ~\ref{compact lattice orbits}), any trajectory of $\set{a_t}_{t\ge 0}$ in $Ax_0$ is dense there. In particular any point in $Ax_0$ is a limit point of some sequence $\pa{a'a_{t_n}}x_0$, for some sequence $t_n\to\infty$.   It follows now from  ~\eqref{c.1}, that 
\begin{equation}\label{c.2}
\overline{A_mx_1}\supset a''u_{32}\pa{s_0}Ax_0=u_{32}(s_1)Ax_0,
\end{equation}
for a suitable choice of $s_1\ne 0$. 
As $A_m$ is closed under multiplication, $\overline{A_mx_1}$ closed under the action of $A_m$. In particular, it follows from ~\eqref{c.2}, that for any $a\in A_m$, $au_{32}(s_1)a^{-1}x_0\in\overline{A_mx_1}$, i.e. $u_{32}(s)x_0\in \overline{A_mx_1}$, where $s$ ranges over the set $\set{e^ts_1:t\ge m}$, which is a ray. The discussion preceding this proof now implies the density of $A_mx_1$ and in particular that $x_1\in\overline{A_mx_1}$ as desired.
\end{proof}
\begin{corollary}\label{cubics are good}
Let $K$ be a totally real cubic number field and let $1,\al,\be$ be a basis of $K$ over $\bQ$. Denote $v=(\al,\be)^t\in\bR^2$. 
Then the lattices $x_v,z_v$, have dense $A$ orbits in $X_3$ and in particular they are $GDP$ by corollary ~\ref{dense imply GDP}.
\end{corollary}
\begin{proof}
Let us denote $\al=\al_1,\be=\be_1$ and let $\al_i,\be_i,i=2,3$, be the other two embeddings of $\al_1,\be_1$ into the reals. Denote $
g_0=c\pa{\begin{array}{lll}
1&\al_1&\be_1\\
1&\al_2&\be_2\\
1&\al_3&\be_3
\end{array}
}$, where $c$ is chosen so that $\det\pa{g_0}=1$. Then, $\bar{g}_0\in X_3$ has a compact $A$-orbit by theorem ~\ref{compact lattice orbits}. It is easy to see that there exists a unique matrix $p\in G_0$ as in lemma ~\ref{density}, such that (recall the notation of $\S\S$~\ref{DA})
\begin{equation}\label{strudel}
pg_0=g_v.
\end{equation}
The reader can easily check that the relevant entries of $p$ must be nonzero. We apply lemma ~\ref{density}, and conclude that $\bar{g}_v=z_v$ has a dense $A$-orbit in $X_3$. In order to see that $x_v$ has a dense orbit, we note that the involution $g\mapsto (g^t)^{-1}=g^*$ of $G_0$ , descends to a diffeomorphism of $X_3$. We denote this map by $\bar{g}\mapsto\bar{g}^*=\overline{g^*}$. This is the well known map which sends a lattice to its dual. Since the group $A$ is invariant under this involution, we have that for any lattice $x$, $\pa{\overline{Ax}}^*=\overline{Ax^*}.$ In particular, $x$ has a dense orbit if and only if $x^*$ has. In a similar way to what we have already shown, one can show that the lattice spanned by the columns of 
$g_1=\pa{
\begin{array}{lll}
1&0&0\\
0&1&0\\
-\al&-\be&1
\end{array}
},$ has a dense $A$-orbit, in $X_3$. As $g_1^*=h_v$, it follows that $\bar{h}_v=x_v$ has a dense orbit too, as desired. 
\end{proof}
\section{Irregular $A$ orbits}\label{irregular orbits}
In this section we use the existence of lattices $x\in X_d$ ($d\ge 3$) for which $\mu(x)=0$ (theorem ~\ref{theorem 101}), and theorem ~\ref{corollary 102}(2) to give examples of lattices in $X_3$ having irregular $A$ orbit closures of a new type. This serves as a counterexample to conjecture 1.1 in \cite{Ma}. Our example proceeds the recent counterexample, F .Maucourant gave to this conjecture in ~\cite{Mau}. Our example is different in nature from Maucourant's example. We use a maximal split torus whilst in \cite{Mau} the acting group does not "separate roots" which seem to be the reason for the abnormality. It still seems plausible that a slightly different version of that conjecture will be true.   
\begin{proof}[Proof of theorem ~\ref{corollary 102}(2)]
We first note that if $x_1,x_2\in X_d$ are commensurable lattices (that is their intersection is of finite index in each), then $\mu(x_1)=0$ if and only if $\mu(x_2)=0$. Let $v\in\bR^{d-1}$ satisfy $\dim_\bQ span\set{1,v_1\dots,v_{d-1}}\le 2$. Then, there exists $\al\in\bR$ and rationals $p_i,q_i, i=1\dots d-1$ such that $v_i=q_i\al+p_i$. Denote $v'=(q_1\al,\dots,q_{d-1}\al)^t$. It follows that $x_v,z_v$ are commensurable to $x_{v'},z_{v'}$ respectively.\\
\textbf{Claim} 1: $\mu(x_{v'})>0$. Working with the definition of $\mu$ we see that it is enough to argue the existence of $d-1$ real numbers $\ga_i$ for which
\begin{equation}\label{mu ind}
N\pa{x_{v'}+(-\ga_1,\dots,-\ga_{d-1},1/2)^t}=\inf_{n\in\bZ}\av{n+1/2}\prod_1^{d-1}\idist{nq_i\al-\ga_i}>0.
\end{equation}
From Davenport's result described in $\S\S$~\ref{dim2} it follows that there exists $\ga_i\in\bR$ such that for each $i$, $\inf_{n\in\bZ} \av{n+1/2}\idist{nq_i\al-\ga_i}>0$. Moreover, if we denote by $m$, a common denominator for the $q_i$'s , then by \cite{Ca2} (theorem 1), we can choose the $\ga_i$'s such that for any $i\ne j$, $\frac{\ga_i}{q_i}-\frac{\ga_j}{q_j}\notin\frac{1}{m}\bZ$. Denote for $r,s\in\bR$ by $\idist{r-s}_m$ the distance modulo $\frac{1}{m}\bZ$ from $r$ to $s$. Denote $\rho=\min_{i\ne j}\idist{\frac{\ga_i}{q_i}-\frac{\ga_j}{q_j}}_m$. Note that for $\eps>0$, $\idist{nq_i\al-\ga_i}<\eps \Rightarrow \idist{n\al-\frac{\ga_i}{q_i}}_m<\frac{\eps}{\av{q_i}}$. Hence if $\max_i\frac{\eps}{\av{q_i}}<\rho/2$, then $\idist{nq_i\al-\ga_i}<\eps$ for at most one index $i$. Let $\eps>0$ be such. Assume that the left hand side of ~\eqref{mu ind} is smaller than $\eps^{d-1}/2$. Then for some $k$, $\idist{nq_k\al-\ga_k}<\eps$ which implies that the left hand side of ~\eqref{mu ind} is $>\eps^{d-2}\inf_{n\in\bZ}\av{n+1/2}\idist{nq_k\al-\ga_k}>0$ as desired.\\
\textbf{Claim 2} $\mu(z_{v'})>0$. We use the notation as in Claim 1. From Davenport's result, we know that there exists $0\ne\ga\in\bR$ such that $\inf_{k\in\bZ}\av{k+1/2}\idist{k\pa{\al/m}-\ga}>0.$ The reader will easily argue the existence of a constant $c>0$ for which $$\forall\vec{n}\in\bZ^{d-1}\setminus\set{0},\quad \prod_1^{d-1}\av{n_i+1/2}\ge c\av{m\vec{q}\cdot\vec{n}+1/2}.$$
Working with the definition of $\mu$ we see that $$
\begin{array}{lll}
\mu(z_{v'})&\ge&N\pa{z_{v'}+(\ga,1/2,\dots,1/2)^t}=\inf_{\vec{n}\in\bZ^{d-1}}\prod\av{n_i+1/2}\idist{\pa{m\vec{q}\cdot\vec{n}}\al/m-\ga}\\
 &\ge&\min\set{\inf_{\vec{n}\cdot\vec{q}\ne 0}c\av{m\vec{q}\cdot\vec{n}+1/2}\idist{\pa{m\vec{q}\cdot\vec{n}}\al/m-\ga}; \frac{\av{\ga}}{2^{d-1}}}>0.
 \end{array}$$ 
\end{proof}
Observe that if $x_n\to x$ in $X_d$, then $\limsup\mu(x_n)\le\mu(x)$. It follows that for $x_0,x\in X_d$ 
\begin{equation}\label{mu inheritance}
 x_0\in\overline{Ax}\Rightarrow \mu(x_0)\ge\mu(x).
 \end{equation}
 By theorem ~\ref{theorem 101}, if $x\in X_d$, ($d\ge 3$) and $\mu(x)>0$, then $Ax$ is not dense.\\
The following statement is a special case of Conjecture 1.1 of ~\cite{Ma}:
\begin{conjecture}[Special case of conjecture 1.1 in ~\cite{Ma}]\label{Margulis conjecture}
For $x\in X_3$, one of the following  three options occurs:
\begin{enumerate}
	\item $Ax$ is dense.
	\item $Ax$ is closed.
	\item $Ax$ is contained in a closed orbit $Hx$ of an intermediate group $A<H<G$, where $H$ could be one of the following three subgroups of $G_0$:
	$$H_1=
	\left(\begin{array}{lll}
	*&*&0\\
	*&*&0\\
	0&0&*
	\end{array}\right), 
	H_2=
	\left(\begin{array}{lll}
	*&0&0\\
	0&*&* \\
	0&*&*
	\end{array}\right), 
	H_3=
	\left(\begin{array}{lll}
	*&0&* \\
	0&*&0\\
	*&0&*
	\end{array}\right).$$
\end{enumerate}
\end{conjecture}
For $t\in\bR$ denote $v_t=(t,t)^t\in\bR^2$. We denote the one parameter group $h_{v_t}$ (recall the notation of \S\S ~\ref{DA}) simply by $h_t$ and the lattice $x_{v_t}$ by $x_t$. 
\begin{theorem}
There exists $t\in\bR$ such that $x_t\in X_3$ violates conjecture ~\ref{Margulis conjecture}.
\end{theorem}
\begin{proof}
By theorem ~\ref{corollary 102}(2), $\mu(x_t)>0$ hence by theorem ~\ref{theorem 101} and ~\eqref{mu inheritance}, possibility (1) is ruled out. To rule out possibilities (2) and (3), we note that if $H$ is either one of the groups $A,H_1,H_2,H_3$, then for any $g\in G_0$,  $H\bar{g}$ is closed in $X_3$ if and only if $g^{-1}Hg$ is defined over $\bQ$. Assume to get a contradiction that for any $t\in\bR$, for $H$ equals one of the above, the group $h_t^{-1}Hh_t$ is defined over $\bQ$. Two elements $g_1,g_2\in G_0$ conjugate $H$ to the same group if and only if $g_1g_2^{-1}$ normalizes $H$. All the above groups are of finite index in their normalizers in $G_0$ and so there exists some $k$ such that whenever $g$ normalizes $H$, then $g^k\in H$. As there are only countably many $\bQ$-groups in $G_0$, there must exist some $t\ne s$ such that $\pa{h_th_s^{-1}}^k=h_{k(t-s)}\in H$ which of course never happens.    
\end{proof}

\end{document}